\documentclass[letterpaper, 10 pt, conference]{ieeeconf}  

\IEEEoverridecommandlockouts                              
\usepackage{amsfonts}
\usepackage{color}
\usepackage{amsmath}
\usepackage{amssymb}
\usepackage{makeidx}
\usepackage{epsfig}
\usepackage{graphicx}

\include{references.bib}

\usepackage{cite}
\usepackage{url}

\newtheorem{theorem}{Theorem}
\newtheorem{proposition}{Proposition}

\newtheorem{corollary}{Corollary}
\newtheorem{remark}{Remark}
\newtheorem{ex}{Example}
\newtheorem{definition}{Definition}

\newtheorem{opq}{Open Question}

\newcommand{\switcheddelay }{{linear feedback system with switched delays}}

\newcommand{\re}{{\mathbb R}}

\newcommand{\n}{{\mathbb N}}

\newcommand{\cq}{{\mathbb{Q}}}

\providecommand{\com[2]}{\begin{tt}[#1: #2]\end{tt}}

\providecommand{\rmj[1]}{{\color{blue}#1}}
\providecommand{\ad[1]}{{\color{red}#1}}

\title{Feedback stabilization of dynamical systems with switched delays$^*$}

\author{Rapha\"el M. Jungers$^1$, Alessandro D'Innocenzo$^2$ and Maria D. Di Benedetto$^2$
\thanks{$^1$ICTEAM Institute, Universit\'e catholique de Louvain, Louvain-la-Neuve, Belgium. Email: {\texttt{raphael.jungers@uclouvain.be}}}
\thanks{$^2$Department of Electrical and Information Engineering, Center of Excellence DEWS, University of L'Aquila, Italy. Email: {\texttt{alessandro.dinnocenzo@univaq.it}}, {\texttt{mariadomenica.dibenedetto@univaq.it}}}
\thanks{$^*$The research leading to these results has received funding from the European Union Seventh Framework Programme [FP7/2007-2013] under grant agreement n°257462 HYCON2 Network of excellence. R.J. is supported by the "Communaut\'e francaise de Belgique - Actions de Recherche Concert\'ees", and by the Belgian Programme on Interuniversity Attraction Poles initiated by the Belgian Federal Science Policy Office.  R.J. is a F.R.S.-FNRS fellow.}
}

\begin{document}

\maketitle
\thispagestyle{empty}
\pagestyle{empty}


\begin{abstract}We analyze a classification of two main families of controllers that are of interest when the feedback loop is subject to switching propagation delays due to routing via a wireless multi-hop communication network. We show that we can cast this problem as a subclass of classical switching systems, which is a non-trivial generalization of classical LTI systems with time-varying delays.

We consider both cases where delay-dependent and delay-independent controllers are used, and show that both can be modeled as switching systems with unconstrained switchings. We provide NP-hardness results for the stability verification problem, and propose a general methodology for approximate stability analysis with arbitrary precision. We finally give evidence that non-trivial design problems arise for which new algorithmic methods are needed.\end{abstract}


\section{Introduction}\label{secIntro}

Wireless networked control systems are spatially distributed control systems where the communication between sensors, actuators, and computational units is supported by a wireless multi-hop communication network. The main motivation for studying such systems is the emerging use of wireless technologies in control systems (see e.g. \cite{akyildiz_wireless_2004}, \cite{SongIECON2010} and references therein) and the recent development of wireless industrial control protocols, such as WirelessHART (\texttt{www.hartcomm2.org}) and ISA-100 (\texttt{www.isa.org}). The use of wireless Multi-hop Control Networks (MCNs) in industrial automation results in flexible architectures and generally reduces installation, debugging, diagnostic and maintenance costs with respect to wired networks. Although MCNs offer many advantages, their use for control is a challenge when one has to take into account the joint dynamics of the plant and of the communication protocol (e.g. scheduling and routing).

Wide deployment of wireless industrial automation requires substantial progress in wireless transmission, networking and control, in order to provide formal models and analysis/design methodologies for MCNs. Recently, a huge effort has been made in scientific research on Networked Control Systems (see e.g.~\cite{Zhang2001},~\cite{WalshCSM2001},~\cite{SpecialIssueNCS2004},~\cite{Andersson:CDC05},~\cite{MurrayTAC2009},~\cite{Hespanha2007},\cite{HeemelsTAC10}, and references therein for a general overview) and on the interaction between control systems and communication protocols (see e.g.~\cite{Astrom97j1},~\cite{walsh_stability_2002},~\cite{yook_trading_2002},~\cite{TabbaraCDC2007},~\cite{TabbaraTAC2007}). In general, the literature on NCSs addresses non--idealities (e.g. quantization errors, packets dropouts, variable sampling and delay, communication constraints) as aggregated network performance variables, neglecting the dynamics introduced by the communication protocols. In~\cite{Andersson:CDC05}, a simulative environment of computer nodes and communication networks interacting with the continuous-time dynamics of the real world is presented. To the best of our knowledge, the first integrated framework for analysis and co-design of network topology, scheduling, routing and control in a MCN has been presented in \cite{AlurTAC11}, where switching systems are used as a unifying formalism for control algorithms and communication protocols. In \cite{DiBenedettoIFAC11Stab}, stabilizability of a MCN has been addressed for SISO LTI plants. In \cite{PappasCDC2011} a MCN is defined as an autonomous system where the wireless network \emph{itself} acts as a fully decentralized controller.

\begin{figure}[ht]
\begin{center}
\includegraphics[width=0.5\textwidth]{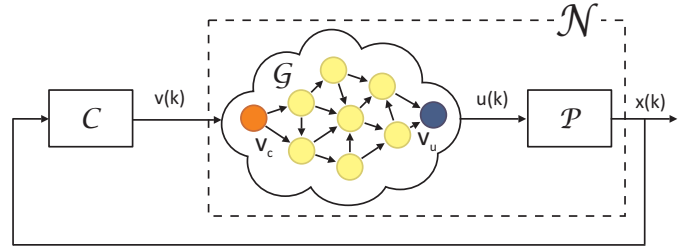}
\caption{State feedback control scheme of a multi-hop control network.}\label{frame}
\end{center}
\end{figure}
In this paper we focus on the effect of routing, and consider a Multi-hop network $\mathcal G$ that provides the interconnection between a state-feedback discrete-time controller $\mathcal C$ and a discrete-time linear plant $\mathcal P$. The network $\mathcal G$ consists of an acyclic graph where the node $v_c$ is directly interconnected to the controller, and the node $v_u$ is directly interconnected to the actuator of the plant. As classically done in multi-hop (wireless) networks to improve robustness of the system with respect to node failures we exploit redundancy of routing paths, namely we assume that the number of paths that interconnect $\mathcal C$ to $\mathcal P$ is greater than one, and that each path is characterized by a delay in forwarding the data. For any actuation data sent from the controller to the plant, a unique routing path of nodes is exploited: since the choice of the routing path usually depends on the internal status of the network, i.e. because of node and/or link failures, we assume that the routing path is non-deterministically chosen. For the above reason, each input signal will be non-deterministically delayed of a finite number of time steps, according to the chosen routing path.

{Our model is strictly related to discrete-time systems with time-varying delay (see e.g. \cite{LiuCTA2006}, \cite{HetelCDC2007} and \cite{ShaoTAC2011}). However the existing results cannot be applied to our model because, due to the routing protocol, control commands generated at different time instants can reach the actuator simultaneously, their arrival time can be inverted, and it is even possible that at certain time instant no control commands arrive to the actuator: these aspects generate switching dynamics that cannot be modeled as a time-varying delay. On the other hand, sufficient stability conditions and LMIs based design procedures have been developed for discrete-time switching systems with time-varying delays (see e.g. \cite{HetelTAC2006} and \cite{ZhangCTA2008}), which are much more general than our model and do not take into account the specific switching structure induced by the routing protocol. In summary, the specific delayed dynamics introduced by the routing in the network cannot be modeled by a generic time-varying delay system: we prove instead that it can be modeled by a pure switching system, where the switching matrices assume a particular form. In this paper, we exploit the characteristic system structure obtained by our network model to derive tailored results that hold for our switching model. We will address the stability and design problem for two families of controllers.}

The first family will be called \emph{delay-dependent controller}, where we assume that the controller is aware of the routing path (namely it can measure the switching signal) and has memory of a finite number of its previous actuation outputs. We prove that the stability analysis problem is NP-hard. Nevertheless, we show that it is possible to compute the worst case growth of the system with arbitrary precision. We prove that for scalar systems the design problem can be solved in a closed form, and provide a counterexample where, even for a 2-dimensional state with 1-dimensional input, there exists a class of switching signals that make the system uncontrollable.

The second family will be called \emph{delay-independent controller}, where we assume that the controller ignores the routing path (namely it ignores the switching signal) and has memory of a finite number of its previous actuation outputs. Differently from the \emph{delay-dependent controller}, in this case it is not trivial to model the closed loop system as a pure switching system. We first prove that this is always achievable, at the cost of augmenting the state space. Thanks to this result, as in the previous case, we show that it is possible to compute the worst case growth of the system with arbitrary precision. We finally provide an example of a scalar system where, depending on the dynamics of the plant, the system can be not stabilizable, stabilizable with memory, and stabilizable without memory. This motivates further studies to develop design methods to guarantee the existence of a stabilizing controller and efficient algorithms to compute it.

From the network point of view, the distinction between the above models depends on the protocol used to route data (see \cite{YangWiMesh2011} and references therein for an overview on routing protocols for wireless multi-hop networks). If the controller node $v_c$ of $\mathcal G$ is allowed by the protocol to chose a priori the routing path (e.g. source routing protocols), then we can assume that the controller is aware of the routing path and the associated delay, and therefore is also aware of the switching signal. If instead the protocol allows each communication node to choose the next destination node according to the local neighboring network status information (e.g. hop-by-hop routing protocols, such as WirelessHART), then we can NOT assume that the controller is aware of the routing path, and therefore also of the switching signal.

The paper is organized as follows. In Section \ref{secProbForm} we provide the problem formulation, and define the two families of controllers introduced above. In Sections \ref{secDelayDependent} and \ref{secDelayIndependent} we address respectively the \emph{delay-dependent} and the \emph{delay-independent} cases. In Section \ref{secConclusions} we provide concluding remarks and open problems for future research.


\section{Problem formulation}\label{secProbForm}
%

In this paper we will address the problem of stabilizing a discrete-time LTI system of the form
$$
x(t+1) = A_P x(t) + B_P u(t),\quad y(t) = x(t),\quad t \geq 0,
$$
with $A_P \in \mathbb R^{n \times n}$ and $B_P \in \mathbb R^{n \times m}$, using a \emph{switching linear controller}\footnote{We will clarify in the following what we mean for \emph{switching linear controller.}} $\mathcal C$. We assume that the control signal $v(t)$ generated by $\mathcal C$ is transmitted to the actuator of the plant $\mathcal P$ via a multi-hop control network \cite{AlurTAC11}. The network $\mathcal G$ consists of an acyclic graph $(V,E)$, where the node $v_c \in V$ is directly interconnected to the controller $\mathcal C$, and the node $v_u$ is directly interconnected to the actuator of the plant $\mathcal P$. In order to transmit each actuation data $v(t)$ to the plant, at each time step $t$ a unique path of nodes that starts from $v_c$ and terminates in $v_u$ is exploited. As classically done in multi-hop (wireless) networks to improve robustness of the system with respect to node failures, we exploit redundancy of routing paths, therefore the number of paths that can be used to reach $v_u$ from $v_c$ is assumed to be greater than one. To each path a different delay can be associated in transmitting data from $v_c$ to $v_u$, depending on the transmission scheduling and on the number of hops to reach the actuator. Since the choice of the routing path usually depends on the internal status of the network (e.g. because of node and/or link failures, bandwidth constraints, security issues, etc.), we assume that the routing path is non-deterministically chosen at each time step $t$. Therefore, the signal $u(t)$ will be non-deterministically delayed of a finite number of time steps, according to the chosen routing path. For the reasons above, we can model the dynamics of a multi-hop control network as follows:

\begin{definition}\label{def-general-sdsystem}
The dynamics of the interconnected system $N$ can be modeled as follows:
\begin{equation}\label{sdsystem}
x(t+1) = A x(t) + B u(v(t - d_{max}:t),\sigma(t - d_{max}:t)),
\end{equation}
where $v(t - d_{max}:t), \sigma(t - d_{max}:t))$ represent the latest $d_{max}$ values of $v(\cdot),\sigma(\cdot).$
\end{definition}
In the above equation,  $\sigma(t)\in D: \,t \geq 0$ is a non-deterministic switching signal, and $A,B$ model the interconnected switching dynamics of the controllability network $\mathcal G_{\mathcal R}$ and of the plant $\mathcal P.$ The set $D=\{d_1,\dots,d_{|D|}\}:\, d_i\in \n, d_{max}=\max{(D)}$ is the set of possible delays introduced by all routing paths.


The above model is quite general, and allows representing a wide range of routing communication protocols for (wireless) multi-hop networks \cite{YangWiMesh2011}. We will prove in the following sections that the model above can be cast as a pure switching system. Therefore, all tools developed for general switching systems can be used for stability analysis and controller design  (e.g. \cite{jungers_lncis,sun-ge}).  However, the particular delay model that we are considering makes our system a special case of general switching systems, endowed with a characteristic matrix structure. In this paper we aim at exploiting such special structure to derive tailored results that hold for our specific switching model. We will show that, already in one of the simplest framework that one could imagine, nontrivial situations and challenging problems occur. As discussed above we think that such results, even if they cannot be applied to general switching systems, have an important impact on analysis and design problems arising in multi-hop control networks.

As discussed in Section \ref{secIntro}, our switching model is a non-trivial generalization of classical LTI systems where the control signal can be subject to a variable delay before arriving to the plant. As we will see below, there are favorable situations where one can design a controller, just like in the LTI case, which stabilizes the system, no matter what delay actually occurs.\\
We emphasize that several variations of this model are possible.  For instance, in our setting, it could happen, during the run of a \switcheddelay, that at some particular time, no feedback signal comes back to the plant. We assume that in this situation the actuation input to the plant is set to zero. In many practical situations, it is also possible to implement a hold that would keep the previous input signal and resend it to the plant at the particular times for which the switching signal implies that no $v(t)$ from the controller is conveyed towards the plant. We defer the comparison of such variants for further studies.

It remains to clarify what we meant by \emph{switching linear controller} in the control scheme definition. In fact, we have to specify the input that the controller receives in order to compute its output. As we will see, there is not a unique straightforward generalization of classical LTI systems, and we make here an important distinction between the situation where the controller knows at each time step $t$ the switching signal $\sigma(t)$ and the situation where it does not. As illustrated in Section \ref{secIntro}, this distinction depends on the protocol used to route data in the network. If the controller node selects a priori the whole routing path up to the actuator node, then we assume that the controller is aware of the switching signal $\sigma(t)$ for each time step $t$. If instead each communication node locally selects the next destination node, then we assume that the controller ignores the switching signal.

For classical (i.e. non-switched) feedback systems with fixed delay, it is well known that the system can be neither controllable nor stabilizable if the feedback only depends on $x(t)$ (that is, if the controller does not have a memory of its previous outputs $v(t'),\ t-d_{max}<t'<t-1$).  Therefore, we take it for granted here that the controller has a memory of its past $d_{max}$ outputs.

\begin{definition}\label{def-dep}\label{eq-ut} A \emph{linear delay-dependent} controller is defined for $\sigma(t)=d$ as:
\begin{equation}\label{eq-vt-dep} v(t)= K(d) \tilde u(t),\end{equation} where  $\tilde u(t) = (x(t),u_{1}(t),u_{2}(t),\dots, u_{d_{max}}(t)),$ \begin{equation}\label{eq-ud} u_{d}(t)=\sum_{t'<t: t'+\sigma(t')=t+d} v(t'),\end{equation} and $K(d)\in \re^{m\times (md_{max}+n)}$. \end{definition}

\begin{definition}\label{def-indep} A \emph{linear delay-independent} controller is defined as:
\begin{equation}\label{eq-vt}  v(t)= K \tilde v(t),\end{equation} where \begin{equation}\label{eq-def-vtilde}\tilde v(t) = (x(t), v(t-d_{max}),\dots,v(t-1))),\end{equation} and $K\in \re^{m\times (md_{max}+n)}$.\end{definition}
The asymmetry between Definitions \ref{def-dep} and \ref{def-indep} deserves an explanation: in Definition \ref{def-dep}, it is assumed that the controller knows the previous values of the switching signal, and then can reconstruct $u_d(t)$ by applying equation (\ref{eq-ud}). In this situation the controller, in order to compute $v(t)$, can use the values of $u_d(t)$, which is the sum of all past control commands $v(t'), t'<t$ that will reach the actuator after $d$ time steps. In Definition \ref{def-indep} however, we suppose that the only available information for the controller is $x(t),$ and this explains why the only variables it can use are $x(t)$ and the $d_{max}$ past control commands $v(t'), t-d_{max} \leq t' \leq t-1$.

\begin{definition}We say that a system is \emph{stable} if, for any switching signal $\sigma(t)$ and for any initial condition $\tilde v(0),$ $$ \lim_{t\rightarrow \infty}\tilde v(t)=0, $$ where $\tilde v(t)$ is as defined in (\ref{eq-def-vtilde}), {and represents an extended state space containing both the state of the plant and the past $d_{max}$ outputs of the controller.}  \end{definition}


\section{Delay-dependent feedback}\label{secDelayDependent}

In this section we show that it may be hard or impossible to analyze stability of a \switcheddelay. We first show how to model this problem as a switching system stability analysis, and then prove that the problem is NP-hard. We then show examples for which such a controller can be designed, and other for which no controller exists.


\subsection{Modeling}

Since the controller is aware of the delays in the delay-dependent framework, we are allowed to use the variables $u_s(t)$ in the state-space, and it is then easy to write the corresponding closed loop system {from Definitions \ref{def-general-sdsystem} and \ref{def-dep}.}  We obtain the following equations: \begin{proposition}\label{prop-delay-dependent-cl}A \switcheddelay{} and delay-dependent controller can be modeled as a switching system with arbitrary switching signal as follows:
 \begin{equation}\label{switched-closed-loop} \tilde u({t+1})= M_{\sigma(t)} \tilde u(t),\quad M_{\sigma(t)} \in \Sigma, \end{equation} where \begin{equation}\label{eq-sigma}\Sigma=\left\{\begin{pmatrix} A&B&0&\dots & 0\\ 0&0&I&\dots & 0\\0&0&0&\ddots&0\\0&0&0&\dots &I\\0&0&0&\dots &0 \end{pmatrix}+ E(d) K(d):d\in D\right \},\end{equation}
and $E(d)$ is the block-column matrix with only zero blocks, except block $(d+1)$ which is equal to the identity if $d\neq0$ and to $B$ if $d=0$. {More precisely, for a same switching signal, the dynamics of \eqref{sdsystem} and \eqref{switched-closed-loop} are equal.}
\end{proposition}

We skip the proof, which is obvious for the delay-dependent case.  As we will see below, the reduction is less obvious for the delay-independent case, and necessitates an increase of the dimension of the state space.


\subsection{Analysis}

As illustrated above, a \switcheddelay{} can be put in the well studied framework of linear switching systems with arbitrary switching signal. Even though these systems have been at the center of a huge research effort in the last decades \cite{sun-ge,LeeD07,jungers_lncis,valcher-positive}, they are known to be very difficult to handle. Nevertheless, it follows from Proposition \ref{prop-delay-dependent-cl} that one can check the stability of a given \switcheddelay{} with arbitrary precision:
\begin{corollary}\label{cor-approx-dep}
For any $\epsilon,$ there exists an algorithm which, given a \switcheddelay{} and delay-dependent feedback, computes in finite time the worst rate of growth of the system, up to an error of $\epsilon.$  More precisely, for any real $r>0$ the algorithm decides in finite time whether
\begin{itemize}
\item $\exists K \in \re : \forall \sigma, \forall t,\ |\tilde u(t)|\leq K (r+\epsilon)^t$;
\item $\exists K \in \re, K>0, \exists \sigma:\ |\tilde u(t)|\geq K (r-\epsilon)^t.$
\end{itemize}
\end{corollary}
\begin{proof}
Proposition \ref{prop-delay-dependent-cl} reformulates the \switcheddelay{} as a classical switching system, thus it is possible to apply one of the classical stability decision procedures derived in \cite[Corollary 3.1]{protasov-jungers-blondel09} or \cite[Theorem 6.1]{ajprhscc11}.
\end{proof}

The above corollary provides a tool to approximate the worst rate of growth, by bisection on $r$. Thus, it is possible to decide with an arbitrary precision whether a \switcheddelay{} and delay-dependent feedback is stable. However, no polynomial-time algorithm is known to solve this problem, and hence the proposed solution does not work in polynomial time.  This is not surprising in view of our next result: we show that given a system and its controller, it is NP-hard to decide whether the controller asymptotically stabilizes the system.
\begin{theorem}
Unless $P=NP,$ there is no polynomial-time algorithm that, given a linear system $A,$ its control matrices $K_i,$ and a set of delays $D,$ decides whether the corresponding delay-dependent \switcheddelay{} is stable.

Also, the question of whether the system remains bounded is turing-undecidable.
 This is true even if the matrices have nonnegative rational entries, and the set of delays is $\{0,1\}.$
\end{theorem}
\begin{proof}
Our proof works by reduction from the matrix semigroup stability, and the matrix semigroup boundedness, which are well known to be respectively NP-hard and Turing undecidable \cite[Theorem 2.4 and Theorem 2.6]{jungers_lncis}. In this problem, one is given a set of two matrices $\Sigma=\{A_1,A_2\}\subset \cq_+^{n\times n}$ ($\cq_+$ is the set of nonnegative rational numbers) and one is asked whether for any sequence $(i_t)_0^\infty,\ i_t\in [1,2],$ the corresponding product $A_{i_1}A_{i_2}\dots A_{i_T}$ converges to the zero matrix when $T\rightarrow \infty$ (respectively remains bounded when $T\rightarrow \infty$).

Let us consider a particular instance $\Sigma=\{A_1,A_2\}\in \cq_+^{n\times n}$ of the matrix semigroup stability (resp. boundedness) problem. We build a delay-dependent \switcheddelay{} whose closed loop system can be written as follows: \begin{eqnarray}\tilde u({t+1})&=&M_i \tilde u(t),\quad M_i\in \Sigma', \end{eqnarray} where  $\Sigma'$ is a set of $2n \times 2n$ matrices. The set $\Sigma'$ is product-bounded (resp. stable) if and only if $\Sigma$ is.

Our construction is as follows: we set $D=\{0,1\}$ as the set of delays, and we build a \switcheddelay{} with a plant of dimension $n$ as follows: the system matrix is given by $$ A=0,\, B=I,$$ and the feedback matrix (in block form) by $$K_d= \begin{pmatrix} A_1 &A_2 \end{pmatrix} $$ for $d=0,1.$  Thus, the corresponding closed loop feedback switching system can be expressed from Proposition \ref{prop-delay-dependent-cl} as $$ \tilde u({t+1})= M_i \tilde u(t)\quad M_i \in \Sigma', $$
where \begin{equation}\label{eq-np-switching}\Sigma'= \left\{ \begin{pmatrix} A_1 &A_2 \\0&0\end{pmatrix},\begin{pmatrix}0&0\\ A_1 &A_2 \end{pmatrix}\right\}.\end{equation} Writing $\tilde u(t) = (x(t), u_1(t))$ we have that, depending on $\sigma(t)$, either $$\tilde u({t+1})= (A_1x(t)+A_2u_1(t),0)$$ or $$\tilde u({t+1})= (0,A_1x(t)+A_2u_1(t)).$$ It is straightforward to see from Equation (\ref{eq-np-switching}) that the set $\Sigma'$ is stable (resp. product bounded) if and only if $\Sigma$ is. Indeed, the blocks in the products of matrices in $\Sigma'$ are arbitrary products of matrices in $\Sigma.$  This concludes the proof.
\end{proof}

\begin{remark}
It is not known (to the best of our knowledge) whether the matrix semigroup stability problem is Turing decidable (say, for matrices with rational entries). Thus, the above proof does not allow us to conclude that the \switcheddelay{} stability problem is undecidable. This is why we only claim that the stability problem is NP-hard, while the boundedness problem is provably Turing undecidable.
\end{remark}


\subsection{Design}

We complete this section by addressing the \emph{design question:} given a \switcheddelay{}, find the actual value of the delay-dependent controllers $K(d)$ such that the resulting system is stable by considering the particular structure exhibited in (\ref{eq-sigma}). As was said above, much less is known in the literature about the design of regular switching systems.

We first completely solve the linear delay-dependent controller design for one-dimensional systems ($n=m=1$). Of course, we have to assume that the system is controllable (i.e. $b\neq 0$), since it must already be the case for a fixed delay.
For classical LTI systems with fixed delay, in the case where the controller has a memory of its past $d_{max}$ outputs, a solution that drives the trajectory exactly onto the origin in finite time is given by an extension of the Ackermann formula: \begin{equation}\label{eq-ack}K_{ack,d}=(-a^{d+1}/b,-a^{d},-a^{d-1},\dots,-a).\end{equation}
It appears that for a \switcheddelay{} too, there is a solution that reaches exactly the origin in finite time $d_{max}$, where $d_{max}$ is the maximum delay:
\begin{theorem}
Consider System (\ref{sdsystem}) with $n=m=1$ and suppose $a,b\neq 0.$ Then, the system is controllable with the following delay-dependent feedback controller: $$K(d)=K_{ack,d_{max}}/(a^{d_{max}-d}),$$ where $K_{ack,d_{max}}$ is the controller as in equation \eqref{eq-ack} for a system with fixed-delay $d_{max}$. Moreover the system reaches exactly zero at latest at time $t=d_{max}+1$.\end{theorem}
\begin{proof}
Let $\tilde u(t) = (x(t),u_{1}(t),u_{2}(t),\dots, u_{d_{max}}(t))$ be the state of the system. From Equation (\ref{eq-ack}) we have, at any time $t,$
\begin{eqnarray}\label{eq-ackerman-zero} \nonumber(a^{d_{max}+1}/b)x(t) + \sum_{s=1}^{d_{max}} a^{d_{max}-s+1}u_{s}(t)&&\\ +K(\sigma(t))\tilde u(t)(a^{d_{max}-{\sigma(t)}})&=&0. \end{eqnarray}
By Definitions \ref{def-dep} and \ref{sdsystem} it follows that for $s=1,\dots,d_{max},$
\begin{eqnarray}\label{eq-us}
u_s(t+1)&=&u_{s+1}(t)\quad \mbox{if }\sigma(t)\neq s\\
\nonumber u_s(t+1)&=&u_{s+1}(t)+v(t)\quad \mbox{if }\sigma(t)= s,
\end{eqnarray}
where we fix for conciseness of notations that $u_{d_{max}+1}=0.$ Observe that the last term in the left-hand side of Equation (\ref{eq-ackerman-zero}) is equal to $v(t)(a^{d_{max}-{\sigma(t)}}).$ Multiplying that equation by $a,$ and making use of Equation \eqref{eq-us}, we obtain:
\begin{align*}
0 =& a^{d_{max}+1}/b(ax(t)+bu_1(t)+bz(t))\\
&+\sum_1^{d_{max}}{a^{d_{max}+1-s}u_s(t+1)}\\
=& (a^{d_{max}+1}/b)x({t+1})+\sum_1^{d_{max}}{a^{d_{max}+1-s}u_s(t+1)} \\
=&K_{ack,d_{max}}\tilde u(t+1)\\
=&v({t+1})(a^{d_{max}-{\sigma(t+1)}}),
\end{align*}
where, again for conciseness, we introduce the variable $z$ such that $z(t)=v(t)$ if $0 \in D$ and $\sigma(t)=0$, and $z(t)=0$ otherwise. In conclusion, if the controller is applied at time $1,$ the output of the controller at time $2$ is $v_{2}=0.$ Thus, by induction, $$\forall t'>1,\, v({t'})=0. $$  This implies (see Equation (\ref{eq-ud})) that $$\forall t''>d_{max},\forall s,\,u_s{(t'')}=0.$$ In turn, since $v_{d_{max+1}}= K_{ack,\sigma({d_{max}+1)}}\tilde u(d_{max}+1)=0,$ this implies that $x_{d_{max}+1}=0.$ \end{proof}
We now show that the situation becomes more complex as soon as the dimension of the plant is equal to $2.$
The following is an example of a system with $n=2,m=1$ that is stabilizable with fixed delays, but not with switching delays.
\begin{ex}
Consider a \switcheddelay{} with the following values:
$$A=\begin{pmatrix}0& 2\\2& 0\end{pmatrix},\quad B=\begin{pmatrix}0& 1\end{pmatrix}^T, $$ $$ D=\{0,1\},\quad \sigma(t)=t \mbox{ mod } 2.$$

That is, $\sigma= 0$ when $t$ is even, and $1$ when $t$ is odd. Then, if $x(0)=(1, 0),$ the system is not controllable. Indeed, one can show by induction that for any even time $t,$ $x_1(t)=2^t.$\end{ex}


\section{delay-independent feedback}\label{secDelayIndependent}

In this section we show that, when the controller does not depend on the delay, the situation is harder because one has to design a single controller that would work for any possible switching signal. Differently from the \emph{delay-dependent controller}, in this case it is not trivial to model the closed loop system as a pure switching system. We first prove that this is always achievable, at the cost of augmenting the state space. This implies, as in the delay-dependent case, that it is possible to compute the worst case growth with arbitrary precision. We finally provide an example of a scalar system where, according to the dynamics of the plant, the system can be not stabilizable, stabilizable with memory, and stabilizable without memory.


\subsection{Modeling}

The delay-independent controller has no access to the variables $u_s(t),$ since one needs the previous values of the switching signal in order to reconstruct these variables. The only variables that the controller can use are its previous outputs $v({t'}):t'<t.$
As done in Proposition \ref{prop-delay-dependent-cl}, the dynamics of the closed loop system can be written as follows:

\begin{equation}\label{switched-closed-loop-independent} \tilde v({t+1})= M(\sigma,t) \tilde v(t)\quad M(\sigma,t) \in \Sigma, \end{equation} where
\begin{eqnarray}\label{eq-sigma-cl-independent} \Sigma&=&\left\{\begin{pmatrix} A&0&0&\dots & 0\\ 0&0&I&\dots & 0\\0&0&0&\ddots&0\\K_0&K_1&K_2 &\dots&K_{d_{max}} \end{pmatrix} \right.  \\&&\left. +\sum_{t': t'+\sigma(t')=t}E(1) B E^T({d_{max}+1-\sigma(t')})\right \}.\nonumber\end{eqnarray}

However, there are two important differences between the systems defined by Equations \eqref{eq-sigma} (delay-dependent) and \eqref{eq-sigma-cl-independent} (delay-independent). First, the set $\Sigma$ in (\ref{eq-sigma-cl-independent}) has a number of matrices that can be exponential in the number of delays $|D|.$ Second, the closed-loop formulation (\ref{switched-closed-loop-independent}) is not a switching system with arbitrary switching signal, as the matrix $M(\sigma,t)$ depends not only on $\sigma(t),$ but also on the values $\sigma(t')$ for $t'<t,$ as one can check in (\ref{eq-sigma-cl-independent}). Thus, this setting seems harder to represent as a pure switching system because of this correlation in the succession of matrices. Even though recent methods based on LMI criteria have been proposed that offer a natural framework for analyzing switching signals described by a regular language (e.g. \cite{ajprhscc11,ajprhscc12,LeeD07}), it would be convenient to have a formulation of the closed loop switching system without any constraint on the switching signal. Indeed, more methods for analyzing switching systems have been designed in the general framework of unconstrained switching. In the following theorem we show that it is always possible to model system \eqref{eq-sigma-cl-independent} as a pure switching system, at the cost of augmenting the state space.
\begin{theorem}\label{thm-ssreduction-independent}
Any $n$-dimensional \switcheddelay{} and delay-independent control of dimension $m$ and set of delays $D$ can be represented as an switching system with arbitrary switches among $|D|$ matrices, characterized by a $(n+2d_{max}m)$-dimensional state space.
\end{theorem}
\begin{proof}
The main idea of the proof is to make use of \emph{both $u_d(t)$ and $v(t)$} in the closed loop state-space representation of the system.  Recall that $u_d(t)$ is the sum of the previous outputs of the controller, that are forecast to arrive at the plant at time $t+d,$ and that $v(t)$ is the output of the controller at time $t.$  See Equations (\ref{eq-ud}) and (\ref{eq-vt}).\\
Of course, the controller does not know $u_d(t)$ in this delay-independent setting (we will take that into account in the construction), but this variable is needed in order to reconstruct the feedback signal.  On the other hand, $v(t)$ is needed in order to represent the memory of the controller.  We now formally describe the state-space, and then the matrices. We set $d_{max}=\max(D)$ as usual.

We define the state-space vector $\tilde w(t)\in \re^{n+2d_{max}m}$ as follows:
$$ \tilde w (t)=(x(t),u_1(t),\dots, u_{d_{max}}(t),v(t-d_{max}),\dots,v(t-1)).$$
Then, for any $d\in D,$ $\tilde w(t+1)$ can be expressed as a linear function of $\tilde w (t).$ Indeed, for any $d\in D,$
the following equations describing the computation of $\tilde w(t+1)$ are linear, and only depend on $\tilde w(t):$
\begin{eqnarray}
\nonumber x(t+1)&=&Ax(t)+Bu_1(t) + B z_{0,d}(t),\\
\nonumber u_s(t+1)&=&u_{s+1}(t) + z_{s,d}(t), \quad 1\leq s\leq d_{max} \\
\nonumber V(t+1)&=& (0,\dots,0,K_0x_{t})^T+\\&&\begin{pmatrix}0&I&\dots&0\\0&0&\ddots&0\\0&\dots&0&I\\K_1 &\dots &&K_{d_{max}}\end{pmatrix}V(t).
\end{eqnarray}
In the above equations, $$ K=(K_0,K_1,\dots,K_{d_{max}})\in \re^{m\times (n+d_{max} m)}$$ is the linear controller, $$V(t)=(v(t-d_{max}),\dots,v(t-1))^T\in\re^{d_{max} m} $$ is the memory of the controller, and $z_{s,d}:\, s=0,\dots, d_{max}$ represents the controller output, to be fed back to the plant with a delay $d:$ \begin{eqnarray}\nonumber z_{s,d}&= &K\begin{pmatrix}x(t)\\V(t) \end{pmatrix} \quad \mbox{ if } s=d,\\\nonumber&=& 0\quad \mbox{otherwise}.\end{eqnarray}
\end{proof}


\subsection{Analysis}

Theorem \ref{thm-ssreduction-independent} implies that, given a \switcheddelay, one can build a corresponding set of matrices, and analyze the stability of the corresponding switching system in order to check for the stability of the given system. In particular the following corollary is equivalent to Corollary \ref{cor-approx-dep}:
\begin{corollary}
For any $\epsilon,$ there exists an algorithm which, given a \switcheddelay{} and delay-independent controller, computes in finite time the worst rate of growth of the system, up to an error of $\epsilon.$
\end{corollary}


\subsection{Design}

We end this section by commenting on the design problem, that is, to find suitable values for the entries in $K$ so that the corresponding system is stable. Since it is harder to control such a system than with a delay-dependent controller, this problem might well have no solution, depending on the set $D$ and the actual value of $A$ and $B.$ Below is a simple example, which shows that several situations can be possible, already in the arguably simplest case $n=m=1,$ and $D=\{0,1\}.$
\begin{ex}\label{ex-indep}
In this example we assume that the controller stores the previous value of $x(t)$ instead of the previous value of $v(t).$  We make this choice for the sake of clarity, in order to have simpler matrices in the equivalent switching system. It is easy to check that in this slightly modified setting, one can still apply the trick of Theorem \ref{thm-ssreduction-independent} and obtain the following three-dimensional switching system as a representation of a one-dimensional \switcheddelay:
\begin{equation}\label{eq-ssreduction-independent}\begin{pmatrix}x({t})\\ x({t+1})\\ u_1(t+1)\end{pmatrix}= M_{\sigma(t)} \begin{pmatrix}x({t-1})\\ x(t)\\ u_1(t)\end{pmatrix}.\end{equation}

where
$$ M_0=\begin{pmatrix}0& 1 &0\\ bk_1 &a+ bk_2& 1\\0& 0 &0 \end{pmatrix}, \quad M_1=\begin{pmatrix}0 &1& 0\\ 0& a& 1\\ bk_1 &bk_2 &0\end{pmatrix}. $$

The controller stores the value of $x({t-1})$ for one iteration. It then makes use of it and of the current value $x(t)$ for computing its output $v(t).$ If the delay is 1 $v(t)$ is put "in the queue" (third entry of the vector), while if the delay is zero it is directly added in the plant in order to compute $x({t+1}).$ Let us consider $b=1,\ D=\{0,1 \}$. Depending on the other values, we obtain the following cases:
\begin{itemize}
\item For $a>3,$ system (\ref{eq-ssreduction-independent}) is unstable, whatever controller $K$ is applied. This can be seen by observing that in this case, $trace(M_1)\geq 3,$ hence $M_1$ is unstable. By Theorem \ref{thm-ssreduction-independent}, the corresponding \switcheddelay{} is uncontrollable.
\item For $a<1$ the system is clearly stabilizable without any controller ($k_1=k_2=0$), since it is the case for the autonomous stable dynamical system.
\item For a=1.1 the system is controllable without using memory (i.e. $k_1=0$), e.g. by taking $k_2=-0.5,$ and the switching system (\ref{eq-ssreduction-independent}) is stable.
\item Finally, for $a=2$ the system is still controllable, but in this case one needs $k_1\neq 0$: indeed, if $k_1=0$ one can restrict himself to the $2 \times 2$ lower-right corner of the matrices, and this subsystem is unstable because $trace(M_0)\geq 2$. On the other hand stabilizing controllers exist with $k_1\neq 0,$ as for instance $k_1=0.4,k_2=-1.5.$
\end{itemize}
In the last two cases, in order to check for the validity of the proposed controller, one can check that the joint spectral radius of the set $\{M_0,M_1\}$ corresponding to the proposed controller is smaller than one, for instance by making use of the JSR toolbox, available on the web \cite{jsr-toolbox}.
\end{ex}


\section{Conclusion}\label{secConclusions}

In this paper, we studied the algorithmic analysis and design of linear discrete-time systems with switched delays. Although many different control strategies are possible, we focused on two simple models and showed that many variate situations can occur. We presented favorable cases, where the problem is algorithmically solvable in polynomial time, and proved that the problem is NP-hard in general. More importantly, we provided an algorithmic procedure to decide stability of a given system with arbitrarily small error (of course, with running time increasing when the error decreases).  Our work raises many natural questions.  We end this paper by mentioning some of them.\\  For the \emph{design questions}, in the one-dimensional case, we showed that the problem is easy in the delay-dependent framework, but the possibility to design a controller in a delay-independent framework seems less straightforward. This leads to our first question:
\begin{opq}What are the particular values of $a,b$ and the sets $D$ in the one dimensional case (see Example \ref{ex-indep}) for which there is a stabilizing delay-independent linear feedback?
Is there an efficient algorithm for deciding the stability/boundedness in this case?\end{opq}
%
%
\begin{opq} Is the stability/boundedness also NP-hard to decide in the delay-independent framework?
\end{opq}
\begin{opq} In our framework, the switching signal might cause the feedback signal to be empty at some times.  We implemented it as a zero signal, but one might for instance implement a hold, which would in this case repeat the previous feedback signal.
What are the situations where the implementation of this hold improves controllability?  How to recognize such situations?
\end{opq}
\begin{opq} How is the situation changed if one is interested in the stability with probability one instead of the worst-case stability?
\end{opq} Suppose that each delay in $D$ appears with a certain probability.  Then, one might only require stability with probability one for System (\ref{sdsystem}).
It is known that the almost sure stability of a classical switching system is ruled by its so-called \emph{Lyapunov Exponent.} Recently, convex optimization techniques have been proposed in order to approximate this quantity \cite{protasov-jungers-lyap}.  So, for the analysis question, these techniques can be applied to the equivalent switching system reformulation that we provided in this paper.  We leave the design question for further work.

\def\cprime{$'$} \newcommand{\noopsort}[1]{} \newcommand{\singleletter}[1]{#1}


\end{document}